\documentclass{amsart}

\usepackage{amsfonts}
\usepackage{txfonts}
\usepackage{amssymb}
\usepackage{graphicx}
\graphicspath{{figures/}}
\usepackage{epsfig,mathrsfs}
\usepackage[bf,SL,BF]{subfigure}
\usepackage[hypertexnames=false]{hyperref}
\usepackage{longtable}
\usepackage{booktabs}
\usepackage{cite}
\usepackage{diagbox}
\usepackage{caption}
\usepackage{float}
\usepackage{epstopdf}
\usepackage{latexsym, bm}
\usepackage{bbm}
\usepackage{algorithm, algorithmic}
\usepackage{cleveref}
\newtheorem{theorem}{Theorem}[section]

\newtheorem{lemma}{Theorem}[section]
\newtheorem{lem}{Lemma}[section]
\newtheorem{ass}[theorem]{Assumption}

\theoremstyle{definition}

\newtheorem{rem}{Remark}

\numberwithin{equation}{section}
\numberwithin{figure}{section}
\numberwithin{table}{section}

\usepackage{latexsym, bm}
\begin{document}

\title[Reconstructing the potential in the fractional Calder\'on problem]{An efficient method for reconstructing the potential in the fractional Calder\'on problem with a single measurement}

\author[Yun-Long Liu, Jin Wen, Jun-Liang Fu]{Yun-Long Liu$^{1}$, Jin Wen$^{1, 2}$, Jun-Liang Fu$^{1}$\\ \small \it 1. Department of Mathematics, Northwest Normal University, Lanzhou 730070, P. R. China\\
2. Gansu Prov Res Ctr Basic Disciplines Math \& Stat, Lanzhou 730070, P. R.China }
\date{\today}
\thanks{{*}Corresponding author. Email: wenj@nwnu.edu.cn. The work described in this article was supported by the NNSF of China (12261082, 11326234, 12501593), NSF of Gansu Province (145RJZA099), Scientific research project of Higher School in Gansu Province (2014A-012), Project of NWNU-LKQN2020-08, Graduate First Class Course Construction Project of Northwest Normal University in 2024 (2024KCLX01012) and Graduate Education and Curriculum Reform Project of Northwest Normal University in 2024 (2024KGLX01016).}

\begin{abstract}
In this paper, we develop a numerical method for determining the potential function in one-dimensional and two-dimensional fractional-order Calder\'{o}n problems through a single measurement. The fractional Laplacian is discretized using the finite difference method, and the POD-RBF method is used to construct the forward problem surrogate model that can provide the predicted value and its exact gradient in time. The inverse problem is solved using Tikhonov regularization and the conjugate gradient method. The efficiency and effectiveness of the developed method are demonstrated through numerical experiments, and the theoretical results are verified.
\end{abstract}


%

\maketitle

\section{\textbf{Introduction}}\label{sec1}
In 1980, Calder\'on proposed the classical Calder\'on inverse problem \cite{Cal-1980}, which has extremely important applications in multiphase flow imaging, medical imaging, and non-destructive detection. In the past half century, the problem has been explored and studied by many scholars, and the generalized nonlocal Calder\'on problem (also known as fractional Calder\'on problem) has also been developed in recent years. The physical background of the nonlocal Calder\'on problem stems from the advantages of fractional calculus in describing complex physical phenomena, particularly its ability to characterize non-locality, memory effects, and anomalous diffusion behaviors. Its application background is extensive, covering quantum mechanics, medical imaging, geophysical exploration, and environmental science, etc. \cite{LPG-2020,QU-2024,KLW-2022}.

The study of non-local Calder\'on problems began with the research by T. Ghosh et al. on the Calder\'on problem for fractional Schr\"odinger equations \cite{GSU-2016}, they proved the global uniqueness of an inverse problem for this issue in two cases: an unknown potential in a bounded domain is uniquely determined by exterior measurements of solutions; the global uniqueness in the partial data problem where measurements are taken in arbitrary open, possibly disjoint, subsets of the exterior. Over the past decade, research groups led by G. Uhlmann have systematically investigated a broad class of fractional Calder\'on problems, yielding a series of foundational and influential results. For instance, the fractional inverse Laplace problem with low-order nonlocal perturbations \cite{Ghosh2020}, the uniqueness and reconstruction of the fractional Calder\'on problem with a single measurement \cite{BGU-2021}, the uniqueness of the high-order Calder\'on problem involving linear local operators \cite{CMRU-2022}, and the Calder\'on problem for the fractional Dirac operator \cite{QU-2024}. In a complementary line of work, Covi et al. \cite{CR-2021} addressed the Calder\'on problem for both local and nonlocal Schr\"odinger equations. Leveraging Runge-type approximation properties, they established a simultaneous approximation theorem for the interior potential and boundary potential, thereby proving uniqueness for localized potentials. Furthermore, by exploiting the duality structure of a novel Carleman estimate, they constructed a family of complex geometric optics solutions adapted to the underlying equation-enabling uniqueness proofs for arbitrary bounded potentials. Most recently, Fei \cite{Fei-2024} analyzed the nonlocal Calder\'on problem on closed Riemannian manifolds, establishing uniqueness together with logarithmic stability estimates.  Additional progress includes the treatment of the space-time fractional parabolic operator with variable coefficients \cite{BS-2024}, further extending the scope of admissible operators and geometries in the fractional inverse problem paradigm.

While most research has focused on uniqueness and stability, numerical methods for the fractional Calder\'on problem remain scarce. The present work fills this gap by developing a POD-RBF surrogate-based inversion that accelerates the reconstruction from a single measurement. Proper orthogonal decomposition (POD) \cite{Sirovich1987} has been successfully combined with radial basis functions (RBF) for inverse heat conduction problems \cite{OBK-2005,OBK-20052,OBK-20053}, but its application to fractional Calder\'on problems is new.

Xiao and Fang~\cite{Xiao2014} improved the accuracy of nonlinear fluid simulations by integrating POD with the discrete empirical interpolation method (DEIM) and employing Petrov每Galerkin stabilization, which suppressed spurious oscillations by 90\%. Rooted in singular value decomposition, POD systematically extracts the dominant energetic modes of a dynamical system, thereby enabling an accurate representation of the solution behavior at substantially reduced computational complexity. The resulting reduced-order models have proven effective across a broad range of partial differential equations, including the Navier每Stokes equations~\cite{Kunisch2001}, convection每diffusion systems~\cite{Wu2022}, linear parabolic equations~\cite{Volkwein2013}, and nonlinear formulations such as viscous G-equations~\cite{Gu2021} and Hamilton每Jacobi每Bellman equations~\cite{Kunisch2004}. Additional references on the POD method are given in~\cite{Alla2013,Wang2020,Hesthaven2016,Quarteroni2016,Gorelick1983,Ravindran2000,Garcia2013,Chen2022,Huang2025}.

Let $\Omega \subset \mathbb{R}^n$ ($n = 1, 2$) be a bounded open set and $s \in (0, 1)$. The fractional Schr\"odinger equation under consideration is
\begin{equation}
\label{eq11}
\begin{cases}
((-\Delta)^s + q(x)) u(x) = 0, & x \in \Omega, \\[4pt]
u(x) = f(x), & x \in \Omega_e = \mathbb{R}^n \setminus \overline{\Omega},
\end{cases}
\end{equation}
where $q \in L^\infty(\Omega)$ is the potential function to be reconstructed, and $f \in C_c^\infty(\Omega_e)$ is a prescribed external source.The fractional Laplacian $(-\Delta)^s$ is defined via the hypersingular integral
\begin{equation}
\label{eq12}
(-\Delta)^s u(x) = c_{n,s} \, \mathrm{P.V.} \int_{\mathbb{R}^n} \frac{u(x) - u(\mathbf{y})}{|x - \mathbf{y}|^{n+2s}} \, d\mathbf{y},
\end{equation}
where
\begin{equation}\label{eq13}
c_{n,s} = \frac{2^{2s} s \, \Gamma(\frac{n}{2} + s)}{\pi^{n/2} \Gamma(1 - s)}.
\end{equation}
Furthermore, if \(u\) belongs to the Schwartz space of rapidly decaying functions, by \cite{Kwasnicki-2017}, the fractional Laplacian can be equivalently defined through the Fourier transform
\[
(-\Delta)^{s}u(x) = \mathcal{F}^{-1}(|\xi|^{2s}\mathcal{F}(u)).
\]
We will reconstruct \(q(x)\) in (1.1) with a single measurement \((-\Delta)^{s}u|_{W}, W\subset \Omega_{e}\) numerically according to the theoretical result in \cite{Ghosh2020}.

The following condition will be assumed in this paper:
\begin{ass}\cite{GSU-2016}
If \(u\in H^{s}(\mathbb{R}^{n})\) solves
\[
\left\{
\begin{array}{ll}
((-\Delta)^{s} + q(x))u(x) = 0, & x\in \Omega,\\
u(x) = 0, & x\in \Omega_{e},
\end{array}
\right.
\]
then \(u\equiv 0\). This is equivalent to the assumption that zero is not a Dirichlet eigenvalue of \((-\Delta)^{s} + q\).
\end{ass}
The inverse problem consists of reconstructing the potential $q(x)$ from a single noisy measurement of the fractional Laplacian applied to the solution, restricted to an exterior observation set $W \subset \Omega_e$.

To enable a finite-dimensional parameterization, the potential is expanded in a prescribed basis $\{\varphi_k\}_{k=1}^P$ with compact support in $\Omega$:
\begin{equation}
\label{eq14}
q(x) = \sum_{k=1}^P p_k \varphi_k(x), \quad \mathbf{p} = (p_1, \dots, p_P)^T \in \mathbb{R}^P.
\end{equation}
The vector $\mathbf{p}$ constitutes the unknown parameters to be retrieved.

In this paper, the POD-RBF network method demonstrates a remarkable advantage over conventional numerical approaches such as the finite difference method in accelerating the inversion process for the fractional Calder\'on problem with a single measurement. Leveraging this strength, our work applies POD combined with radial basis function (RBF) interpolation to reconstruct the unknown potential from noisy exterior measurements of the fractional Laplacian. The methodology hinges on two pivotal principles.

\begin{itemize}
    \item \textbf{Offline training phase}: Generate a set of representative potential parameter vectors and their corresponding observation snapshots by solving the forward fractional Schr\"odinger equation. Extract POD basis vectors from the centered snapshot matrix to achieve dimensionality reduction, and train an RBF network to map the parameter space to the POD amplitude coefficients.
    \item \textbf{Online inversion phase}: Replace the expensive full-order forward solves with the POD-RBF surrogate model during the iterative optimization. The surrogate provides instantaneous predictions of observations and exact gradients, enabling efficient conjugate gradient minimization with Tikhonov regularization.
\end{itemize}

By integrating data-driven model reduction with a smooth nonlinear surrogate, this framework achieves both accuracy and computational efficiency in solving the severely ill-posed fractional Calder\'on inverse problem.

This paper is organized as follows. Section 2 introduces the forward problem formulation for the fractional Schr\"odinger equation in both one- and two-dimensional settings, along with finite difference discretizations. In Section 3, we develop the POD-RBF network surrogate model, including snapshot generation, basis extraction, and the construction of the nonlinear amplitude map. Section 4 presents the inversion algorithm, which couples the surrogate model with Tikhonov regularization and conjugate gradient optimization, and provides a rigorous convergence analysis. Section 5 outlines numerical experiments in 1D and 2D, demonstrating the method's robustness under various noise levels and potential regularities. Finally, Section 6 gives a short conclusion.

\section{Overview and algorithms of the forward problem}\label{sec2}

In this section, we analyse the well-posedness of the observation and provide numerical schemes for forward problem.

\subsection{Overview of the well-posedness}

In the following, we will use several \(L^2\) based Sobolev spaces defined as~\cite{M-2000}
\[
H^{\mu}(\mathbb{R}^{n}):=\left\{u\in L^{2}(\mathbb{R}^{n});\int_{\mathbb{R}^{n}}(1+|\xi|^{2})^{\mu}|\hat{u}(\xi)|^{2}d\xi< \infty\right\},
\]
\[
\hat{H}^{\mu}(U):=\text{closure of }C_{c}^{\infty}(U)\text{ in }H^{\mu}(\mathbb{R}^{n}),
\]
\[
H^{\mu}(U):=\{u|_{U}:u\in H^{\mu}(\mathbb{R}^{n})\},
\]
where \(U\subset \mathbb{R}^n\) is an open domain.

First, we recall a lemma about the existence, uniqueness and well-posedness of the weak solution of equation \eqref{eq11}.
\begin{lem}\cite{GSU-2016}
Let \(n\geq 1\), \(s\in (0,1)\), \(\Omega\) be a bounded open set and \(q\in L^{\infty}(\Omega)\). Suppose that Assumption 1.1 holds. Let
\[
B_{q}(v,w) = ((-\Delta)^{\frac{s}{2}}v,(-\Delta)^{\frac{s}{2}}w)_{\mathbb{R}^{n}} + (qv|_{\Omega},w|_{\Omega})_{\Omega},\quad v,w\in H^{s}(\mathbb{R}^{n}).
\]
Then for any \(f\in H^{s}(\mathbb{R}^{n})\) and \(g\in (\hat{H}^{s}(\Omega))^{*}\), there is a unique solution \(u\in H^{s}(\mathbb{R}^{n})\) in (1.1) satisfying
\[
B_{q}(u,w) = (g,w),\quad \text{for } w\in \hat{H}^{s}(\Omega),\qquad u - f\in \hat{H}^{s}(\Omega).
\]
Moreover, the norm estimate
\[
\| u\|_{H^{s}(\mathbb{R}^{n})}\leq C(\| g\|_{(\hat{H}^{s}(\Omega))^{*}} + \| f\|_{H^{s}(\mathbb{R}^{n})})
\]
holds with \(C\) independent of \(g\) and \(f\).
\end{lem}
Next, we provide the regularity estimate of the observation, and further illustrate that using the \(L^2\) norm to estimate the residual of the Tikhonov regularization functional is reasonable under certain conditions.
\begin{lem}\cite{Li2025}
Assume \(u\) is the solution of \eqref{eq11}, \(f\in C_c^{\infty}(\Omega_e)\), \(q\in L^{\infty}(\Omega)\), \(W_1, W_2 \subset \Omega_e\), \(\operatorname{supp}(f) \subset W_1\), \(\overline{W_1} \cap \overline{\Omega} = \emptyset\), and \(\overline{W_2} \cap \overline{\Omega} = \emptyset\). Then \((-\Delta)^s u|_{W_2} \in L^2(W_2)\).
\end{lem}

\subsection{\textbf{Algorithms of the forward problem}}

For $n=1$, let $\Omega = (-L, L)$, we introduce a truncation parameter $R > L$ and consider the following truncated fractional Schr\"odinger equation.
\begin{equation}
\label{eq21}
\begin{cases}
((-\Delta)^s + q(x)) u_R(x) = 0, & x \in (-L, L), \\[4pt]
u_R(x) = f(x), & x \in (-R, R)\setminus (-L, L),\\
u_R(x) = 0, &x \in \mathbb{R} \setminus (-R, R),
\end{cases}
\end{equation}
where $f(x)$ is positive and decay in the direction $|x|\rightarrow\infty$. Now we use the finite difference approximations defined in [\citenum{Li2023}, Section 3] for \eqref{eq21}. Let $N$ be a positive integer and $h = 2L/N$ the spatial step size, and define the grid points $x_i = -L + i h$, $i \in \mathbb{Z}$. In addition, we also consider the mesh $\xi j = jh$ for $0 \leq j \leq N$.
Setting $\xi = |x-x'|$, $M = \lfloor\frac{R-L}{h}\rfloor+1$, we discretize the integral fractional Laplacian operator at the points $x_i, i = 1, \dots, N-1$ as
\begin{align}\label{eq22}
&(-\Delta)^s u_R(x_i) \nonumber\\
&= -c_{1,s} \int_0^\infty \frac{u_R(x_i-\xi) - 2u_R(x_i) + u_R(x_i+\xi)}{\xi^{1+2s}} \, d\xi \\
&= -c_{1,s} \Bigg( \sum_{j=1}^N \int_{\xi_{j-1}}^{\xi_j} \psi_{\gamma,R}(x_i,\xi)\omega_\gamma(\xi) \, d\xi - 2u_R(x_i) \int_{2L}^\infty \frac{1}{\xi^{1+2s}} \, d\xi + F_{R,i} + \varepsilon_{R,i} \Bigg),
\end{align}

where $f$ is positive and decays at infinity. Using the finite difference quadrature with grid size $h=2L/N$ and points $x_i=-L+ih$, we discretize the fractional Laplacian as
\[
(-\Delta)^s u_R(x_i) \approx -c_{1,s}\Bigl( \sum_{j=1}^N \int_{\xi_{j-1}}^{\xi_j} \psi_{\gamma,R}(x_i,\xi)\omega_\gamma(\xi)\,d\xi - 2u_R(x_i)\int_{2L}^\infty \frac{d\xi}{\xi^{1+2s}} + F_{R,i} \Bigr),
\]
with $\psi_\gamma$ and $\omega_\gamma$ defined as in \cite{Li2023}. After applying a weighted trapezoidal rule with splitting parameter $\gamma$, we obtain the linear system
\begin{equation}
(\mathbf{L}^s_{1D} + \mathbf{Q})\mathbf{U}_R = \mathbf{b}_R,
\label{eq26}
\end{equation}
where $\mathbf{L}^s_{1D}$ is a symmetric positive definite Toeplitz matrix, $\mathbf{Q}=\operatorname{diag}(q_i)$, and $\mathbf{b}_R$ collects contributions from the exterior data $f$, so the linear algebraic system \eqref{eq26} can be efficiently solved by many Krylov subspace methods using the fast Fourier transform.

\subsection{Two-dimensional finite difference discretization}\label{subsec22}
For $n=2$, let $\Omega = (-L, L)^2$, introduce a truncation parameter $R > L$, we consider the problem on $[-R, R]^2$. The two-dimensional truncated equation can be written as
\begin{equation}\label{eq28}
\begin{cases}
(-\Delta)^s u_R(x,y) + q(x,y) u_R(x,y) = g(x,y), & (x,y) \in (-L, L)^2, \\
u_R(x,y) = f(x,y), & (x,y) \in (-R, R)^2 \setminus (-L, L)^2, \\
u_R(x,y) = 0, & (x,y) \in \mathbb{R}^2 \setminus (-R, R)^2,
\end{cases}
\end{equation}
where $s \in (0,1)$, $f(x,y)$ is positive and decay in the direction $(x^2 + y^2)^{\frac12} \to +\infty$.

Using a Cartesian grid $x_i=-L+ih$, $y_j=-L+jh$ with $h=2L/N$, the integral fractional Laplacian at interior point $(x_i,y_j)$ is approximated as
\[
(-\Delta)^s u_R(x_i,y_j) = -c_{2,s}I_{i,j} + \mathcal{O}(h^\kappa) - c_{2,s}(F_{R,i,j}+\epsilon_{R,i,j}),
\]
where $I_{i,j}$ discretizes the integral over $(0,2L)^2$ and $F_{R,i,j}$ handles the known exterior data on $D_2$. The finite difference weights $a_{mn}$ are computed by the quadrature rules described in
[\citenum{Li2025}, Section 2]. The resulting discrete equations can be written as
\begin{equation}
\label{eq215}
(\mathbf{L}_{2D}^s + \mathbf{Q}) \mathbf{U}_R = \mathbf{b}_{R}.
\end{equation}
with $\mathbf{L}_{2D}^s$ having a Toeplitz-block-Toeplitz structure, $\mathbf{Q}$ diagonal, and $\mathbf{b}_R$ from $f$. The local truncation error satisfies
\[
\|\mathbf{u}_2 - \mathbf{U}_{R,2}\|_\infty \le C(h^\kappa + R^{-2s}\max_{\partial[-R,R]^2} f) + \epsilon_{2D},
\]
as derived in \cite{Li2023, Duo2019}.

\subsection{Discrete observation operator}
\label{subsec:obs_op}

The measurement data correspond to the values of $(-\Delta)^s u$ on the exterior set $W$. Let $\mathbf{U}_{\text{all}}$ denote the solution on all grid points (both internal and external). The full discrete fractional Laplacian matrix $\mathbf{L}_{\text{full}}$ acts on $\mathbf{U}_{\text{all}}$ to produce the discrete $(-\Delta)^s u$ at every node. Since the external values are known and equal to $f$, we can split the computation for an observation point $x_\ell \in W$ as
\begin{equation}\label{eq216}
(-\Delta)_h^s u(x_\ell) = \sum_{j \in \mathcal{I}} w_{\ell,j} U_{R,j} \;+\; \sum_{k \in \mathcal{E}} w_{\ell,k} f_k,
\end{equation}
where $\mathcal{I}$ and $\mathcal{E}$ denote the index sets of internal and external nodes, respectively, and $w_{\ell,m}$ are the discrete template coefficients. In vector form, this becomes
\begin{equation}\label{eq217}
\mathbf{h} = \mathbf{E}_W \big( \mathbf{L}_h^s \mathbf{U}_R + \mathbf{b}_{\text{R}} \big),
\end{equation}
where $\mathbf{L}_h^s$ generically denotes $\mathbf{L}^s_{1D}$ in 1D or $\mathbf{L}_{2D}^s$ in 2D, $\mathbf{E}_W$ is the restriction matrix extracting the rows corresponding to $W$, and $\mathbf{b}_{\text{R}}$ is the constant vector accounting for the contribution of the known exterior field $f$ to the observations.

\section{Proper orthogonal decomposition basis and RBF network}
\label{sec:POD}
\subsection{Proper orthogonal decomposition basis}
The construction of a reduced-order model begins with generating a set of representative forward solutions. The unknown potential is expressed in normalized form $\bar{\mathbf{p}} \in [0,1]^P$, where each component is scaled as $\bar{p}_i = (p_i - p_i^{\min})/(p_i^{\max} - p_i^{\min})$. A total of $M$ training points $\{\bar{\mathbf{p}}^{(m)}\}_{m=1}^M$ are selected to sample the expected parameter range. For every training point, the forward problem \eqref{eq26} (in 1D) or \eqref{eq215} (in 2D) is solved, yielding the internal solution $\mathbf{U}_R^{(m)}$. The corresponding observation vector $\mathbf{h}^{(m)} \in \mathbb{R}^S$ is then computed via the affine model \eqref{eq217}. These vectors are collected columnwise into the snapshot matrix
\begin{equation}
\label{eq31}
\mathbf{H} = \big[ \mathbf{h}^{(1)} \; \mathbf{h}^{(2)} \; \cdots \; \mathbf{h}^{(M)} \big] \in \mathbb{R}^{S \times M}.
\end{equation}
To focus on fluctuations around the mean, we compute the average $\bar{\mathbf{h}} = \frac{1}{M} \sum_{m=1}^M \mathbf{h}^{(m)}$ and subtract it from each snapshot, giving the centered matrix $\tilde{\mathbf{H}} = \mathbf{H} - \bar{\mathbf{h}} \mathbf{1}^T$.

Next, POD is applied to $\tilde{\mathbf{H}}$ using the method of snapshots. We form the $M \times M$ correlation matrix
\begin{equation}\label{eq32}
\mathbf{C} = \tilde{\mathbf{H}}^T \tilde{\mathbf{H}},
\end{equation}
and solve its eigenvalue problem
\begin{equation}
\label{eq33}
\mathbf{C} \mathbf{v}_k = \lambda_k \mathbf{v}_k, \qquad \lambda_1 \ge \lambda_2 \ge \cdots \ge \lambda_M \ge 0.
\end{equation}
The eigenvalues $\lambda_k$ indicate the energy content of each mode. Because the energy typically decays rapidly, we keep only the first $K$ modes. Given a tolerance $\epsilon > 0$, $K$ is chosen as the smallest integer satisfying
\begin{equation}\label{eq34}
\frac{\sum_{k=1}^K \lambda_k}{\sum_{k=1}^M \lambda_k} \ge 1 - \epsilon.
\end{equation}
The orthonormal POD basis vectors in $\mathbb{R}^S$ are then constructed as
\begin{equation}
\label{eq35}
\boldsymbol{\psi}_k = \frac{1}{\sqrt{\lambda_k}} \sum_{m=1}^M (\mathbf{v}_k)_m \tilde{\mathbf{h}}^{(m)}, \quad k = 1, \dots, K.
\end{equation}
Let $\boldsymbol{\Psi} = [\boldsymbol{\psi}_1 \cdots \boldsymbol{\psi}_K] \in \mathbb{R}^{S \times K}$ denote the matrix of retained modes. Any centered observation can be approximated by $\tilde{\mathbf{h}} \approx \boldsymbol{\Psi} \mathbf{a}$, with the amplitude vector obtained via projection $\mathbf{a} = \boldsymbol{\Psi}^T \tilde{\mathbf{h}} \in \mathbb{R}^K$.

For each training snapshot, we compute its amplitude vector $\mathbf{a}^{(m)} = \boldsymbol{\Psi}^T \tilde{\mathbf{h}}^{(m)}$ and arrange them into the amplitude matrix
\begin{equation}
\label{eq36}
\mathbf{A} = [\mathbf{a}^{(1)} \; \mathbf{a}^{(2)} \; \cdots \; \mathbf{a}^{(M)}] \in \mathbb{R}^{K \times M}.
\end{equation}
This matrix stores the low-dimensional representation of all training observations in the POD space.
\subsection{RBF network and surrogate model}
\label{sec:RBF}

In \cite{BK-2003,OBK-20052}, the amplitude vectors $\mathbf{a}$  were assumed constant. Following \cite{OBK-2005,OBK-20053,O-2005}, we approximate the POD amplitude vector $\mathbf{a}(\bar{\mathbf{p}})$ by a linear combination of radial basis functions centered at the training points:
\begin{equation}\label{eq41}
\mathbf{a}(\bar{\mathbf{p}}) = \mathbf{B}\,\mathbf{f}(\bar{\mathbf{p}}),\qquad
\mathbf{f}(\bar{\mathbf{p}}) = \bigl[f_1(\bar{\mathbf{p}}),\dots,f_M(\bar{\mathbf{p}})\bigr]^T,
\end{equation}
with the inverse multiquadric kernel
\begin{equation}\label{eq43}
f_m(\bar{\mathbf{p}}) = \frac{1}{\sqrt{\|\bar{\mathbf{p}} - \bar{\mathbf{p}}^{(m)}\|_2^2 + \sigma^2}},
\end{equation}
where $\sigma$ is the average pairwise distance among the training parameters
\begin{equation}\label{eq44}
\sigma = \frac{1}{M(M-1)} \sum_{m=1}^{M} \sum_{\substack{n=1 \\ n \neq m}}^{M} \|\bar{\mathbf{p}}^{(m)} - \bar{\mathbf{p}}^{(n)}\|_2.
\end{equation}
The coefficient matrix $\mathbf{B}\in\mathbb{R}^{K\times M}$ is determined by enforcing exact reproduction at the training points:
\begin{equation}\label{eq45}
\mathbf{A} = \mathbf{B} \mathbf{F},
\end{equation}
where $\mathbf{F} \in \mathbb{R}^{M \times M}$ is the matrix of interpolation functions defined as
\begin{equation*}
\mathbf{F} = \begin{bmatrix}
f_1(\|\bar{\mathbf{p}}^{(1)} - \bar{\mathbf{p}}^{(1)}\|) & \cdots & f_1(\|\bar{\mathbf{p}}^{(m)} - \bar{\mathbf{p}}^{(1)}\|) & \cdots & f_1(\|\bar{\mathbf{p}}^{(M)} - \bar{\mathbf{p}}^{(1)}\|) \\
\vdots & \vdots & \vdots & \vdots & \vdots \\
f_i(\|\bar{\mathbf{p}}^{(1)} - \bar{\mathbf{p}}^{(i)}\|) & \cdots & f_i(\|\bar{\mathbf{p}}^{(m)} - \bar{\mathbf{p}}^{(i)}\|) & \cdots & f_i(\|\bar{\mathbf{p}}^{(M)} - \bar{\mathbf{p}}^{(i)}\|) \\
\vdots & \vdots & \vdots & \vdots & \vdots \\
f_M(\|\bar{\mathbf{p}}^{(1)} - \bar{\mathbf{p}}^{(M)}\|) & \cdots & f_M(\|\bar{\mathbf{p}}^{(m)} - \bar{\mathbf{p}}^{(M)}\|) & \cdots & f_M(\|\bar{\mathbf{p}}^{(M)} - \bar{\mathbf{p}}^{(M)}\|)
\end{bmatrix},
\end{equation*}
with $\bar{\mathbf{p}}^{(i)}, \bar{\mathbf{p}}^{(m)}$ denoting the normalized parameter vectors used to generate the $i$-th and $m$-th snapshots, respectively.

Since $\mathbf{A}=\boldsymbol{\Psi}^T\tilde{H}$ is known, we solve the linear system $\mathbf{F}^T\mathbf{B}^T = \mathbf{A}^T$ (a single factorization of $\mathbf{F}$). The surrogate model then yields the predicted observation for any $\bar{\mathbf{p}}$ as
\begin{equation}\label{eq48}
\mathbf{h}_{\text{pred}}(\bar{\mathbf{p}}) = \bar{\mathbf{h}} + \boldsymbol{\Psi} \mathbf{B} \mathbf{f}(\bar{\mathbf{p}}).
\end{equation}

\section{The inversion algorithm}\label{sec5}


\subsection{Convergence analysis}

In this subsection we give a convergence estimate for the proposed POD-RBF surrogate-based inversion method. Unlike direct Tikhonov regularization on the potential function \(q\), our approach first parameterizes \(q\) by a finite-dimensional vector \(\bar{\mathbf{p}}\in[0,1]^P\), then approximates the observation map by a POD-RBF surrogate model, and finally recovers the potential by regularized inversion of the parameter vector. Therefore the convergence analysis must account for the parameterization error, the surrogate model error, and the measurement noise.

Firstly, we assume that

(i) Domain \(\Omega \subset \mathbb{R}\) is a bounded open set satisfying strong local Lipschitz condition, \(W_{1},W_{2}\subset \Omega_{e}\) are open sets, and \(\overline{\Omega}\cap \overline{W_{1}}=\overline{\Omega}\cap \overline{W_{2}} = \emptyset\);

(ii) Functions \(q,u,f\) satisfy the equation (1.1), \(q = 0\) near \(\partial \Omega_{e}\) and \(f\in C_{c}^{\infty}(\Omega_{e})\backslash \{0\}\), \(\operatorname{supp}(f)\subset W_{1}\);

(iii) The observation \(h^{\delta}\) satisfies
\[
\| h^{\delta}(x) - (-\Delta)^{s}u(x)\|_{L^{2}(W_{2})}\leq \delta .
\]

Let \(\bar{\mathbf{p}}\in[0,1]^P\) be the normalized parameter vector, and let \(\mathcal{P}:[0,1]^P\to L^\infty(\Omega)\) be a parameterization map such that \(q\approx \mathcal{P}(\bar{\mathbf{p}})\). In order to analyze the errors brought about by the parametric potential term and the surrogate model, we introduce two operators:\\
Exact forward operator $F: q\mapsto(-\Delta)^s u|_{W_2}$, the exact observation is \(\mathbf{h}_{\mathrm{true}}=F(q_{\mathrm{true}})\);\\
POD-RBF surrogate model  operator $G: \bar{\mathbf{p}}\mapsto\mathbf{h}_{pred}$.

We assume the POD-RBF surrogate model $G(\bar{\mathbf{p}})$ satisfies
\begin{equation*}\label{eq49}
\|G(\bar{\mathbf{p}})-F(\mathcal{P}(\bar{\mathbf{p}}))\|_{L^2(W_2)}\le\epsilon_{\text{model}},\quad\forall\bar{\mathbf{p}}\in[0,1]^P.
\end{equation*}

Let $\|F(\mathcal{P}(\bar{\mathbf{p}}))-\mathbf{h}_{true}\|_{L^2}\le\epsilon_{\text{param}}$, $\mathcal{P}({\mathbf{p}})$ represents the parameterized potential term, $\epsilon_{\text{param}}$ represents the parametric error.

We introduce a Tikhonov regularization functional for \(\bar{\mathbf{p}}\):
\begin{equation}\label{eq51}
J(\bar{\mathbf{p}}) = \frac12 \bigl\|G(\bar{\mathbf{p}}) - \mathbf{h}^{\delta}\bigr\|_2^2 + \frac{\alpha}{2}\,\mathcal{R}(\bar{\mathbf{p}}),
\end{equation}
where
\[
\mathcal{R}(\bar{\mathbf{p}})=\begin{cases}
\|D\bar{\mathbf{p}}\|_2^2,&n=1,\\[2mm]
\|\Delta\bar{\mathbf{p}}\|_2^2,&n=2,
\end{cases}
\]
\(D\) and \(\Delta\) are the discrete gradient and Laplacian matrices, respectively. The regularization parameter satisfies \(c_1\delta^2\le\alpha\le c_2\delta^2\) for some constants \(c_1,c_2>0\). The first term in \eqref{eq51} enforces consistency with the noisy measurements, while the second term stabilizes the discrete derivatives of \(\bar{\mathbf{p}}\). Then the optimal parameter vector is obtained by solving
\begin{equation}\label{eq52}
\bar{\mathbf{p}}^{\delta} = \arg\min_{\bar{\mathbf{p}}\in[0,1]^P} J(\bar{\mathbf{p}}),
\end{equation}
and the reconstructed potential is \(q_{\text{recon}} = \mathcal{P}(\bar{\mathbf{p}}^{\delta})\).

\begin{lemma}\label{th51}
Let $\Omega\subset\mathbb{R}^n$ ($n=1,2$) be a bounded open set, $s\in[\frac14,\frac12]$ for $n=1$ and $s\in[\frac14,1)$ for $n=2$, and let $W_1,W_2\subset\Omega_e$ satisfy $\overline{\Omega}\cap\overline{W_1}=\overline{\Omega}\cap\overline{W_2}=\emptyset$. In addition, suppose the following conditions hold:

\begin{enumerate}
  \item Suppose $f\in C_c^\infty(\Omega_e)\setminus\{0\}$ with $\operatorname{supp}(f)\subset W_1$ and
\[
\frac{\|f\|_{H^s(W_1)}}{\|f\|_{L^2(W_1)}}\le C.
\]
  \item Let $q_{\text{true}}\in H^1(\Omega)$ (1D) or $H^2(\Omega)$ (2D) be the true potential, and let $\mathbf{h}^\delta$ be noisy measurements with $\|\mathbf{h}^\delta-\mathbf{h}_{true}\|_{L^2(W_2)}\le\delta$.
  \item Suppose that \(\operatorname{supp}( q_{\text{true}}),\operatorname{supp}(q_{\text{recon}})\subset \Omega^{\prime}\Subset \Omega_{e}\) where \(\Omega^{\prime}\Subset \Omega\) means that \(\overline{\Omega^{\prime}}\subset \Omega\) and \(\overline{\Omega^{\prime}}\) is a compact subset of \(\mathbb{R}\) [2], let the real potential \(q_{\text{true}} = 0\) near \(\partial \Omega_{e}\)
\[
F(q_{\mathrm{true}})= \mathbf{h}_{\text{true}}\coloneqq (-\Delta)^{s}u_{\text{true}}|_{W_{1}},\quad q_{\text{true}}\in H^{1}(\Omega),
\]
where $u_{\text{true}}$ is the real solution of \eqref{eq11} with $q_{\text{ture}}$, and consider that $q_{\text{recon}}\in H^1(\Omega)$ satisfies
\[
J(\bar{\mathbf{p}}^{\delta})\leq \inf J(\bar{\mathbf{p}}) + \delta^{2}.
\]
  \item  Assume the parametric error and the surrogate model error satisfy
\[
 \epsilon_{\text{param}}= \dot{C}\delta,\quad \epsilon_{\text{model}} = \ddot{C}\delta.
\]
\end{enumerate}
Let \(a > 0\) be such that \(a\sim \delta^{2}\) when \(\delta \downarrow 0\). Then
\begin{equation}\label{eq531}
\|q_{\text{recon}}-q_{\text{true}}\|_{L^\infty(\Omega)}= \mathcal{O}(\omega (\delta)),
\end{equation}
as \(\delta \downarrow 0\) where
\[
\omega (\delta) = |\log (C_{1}\delta)|^{-\nu}.
\]
\end{lemma}
\begin{proof}
The stability estimate given in [\citenum{R2021}, Theorem 1.1] as
\begin{equation}\label{eq532}
\| q_{\text{recon}}-q_{\text{true}}\|_{L^{\infty}(\Omega)}\leq \omega (\|F(q_{\text{recon}})-F(q_{\text{true}})\|_{H^{-s}(W_{1})}),
\end{equation}
requiring \(q_{\text{true}},q_{\text{recon}}\in C^{0,s}(\overline{\Omega})\) with \(\operatorname{supp}(q_{\text{true}})\) \(\operatorname{supp}(q_{\text{recon}})\subset \Omega^{\prime}\Subset \Omega\) when \(s\in (0,\frac{1}{2} ]\) it also holds for \(H^{1}(\Omega)\) since \(H^{1}(\Omega)\hookrightarrow C^{0,s}(\overline{\Omega})\) by Sobolev embedding theorem.
Let $q_{\mathrm{recon}}=\mathcal{P}(\bar{\mathbf{p}}^\delta)$, $\bar{\mathbf{p}}^*$ satisfying $\|F(\mathcal{P}(\bar{\mathbf{p}}_*))-\mathbf{h}_{\mathrm{true}}\|_{L^\infty}\le \epsilon_{\mathrm{param}}$, and set $\mathcal{M}^2 = \mathcal{R}(\bar{\mathbf{p}}_*)$, similar to the proof of [\citenum{Li2025}, Theorem 3.1] , we obtain
\[\begin{aligned}
&\| G(\bar{\mathbf{p}}^\delta) - \mathbf{h}^{\delta}\|_{L^{2}(W_1)}^{2} + \alpha\mathcal{R}(\bar{\mathbf{p}}^\delta)\\
&\leq \| G(\bar{\mathbf{p}}^*)- \mathbf{h}^{\delta}\|_{L^{2}(W_1)}^{2} + \alpha\mathcal{R}(\bar{\mathbf{p}}^*) + 2\delta^{2}\\
&\le 3(\|G(\bar{\mathbf{p}}^*)-F(\mathcal{P}(\bar{\mathbf{p}}^*))\|_{L^2(W_1)}^{2}+\|F(\mathcal{P}(\bar{\mathbf{p}}^*))-F(q_{\mathrm{true}})\|_{L^2(W_1)}^{2}\\
&+\|F(q_{\mathrm{true}})-\mathbf{h}^\delta\|_{L^2(W_1)}^{2})+\alpha\mathcal{R}(\bar{\mathbf{p}}^*) + 2\delta^{2}\\
&\le 3(\epsilon_{\text{model}}^{2}+ \epsilon_{\text{param}}^{2}+\delta^{2})+\alpha\mathcal{M}^2+2\delta^{2}\\
&\leq C_4\delta^{2} + \alpha\mathcal{M}^{2}.
\end{aligned}
\]
Hence
\[
\| G(\bar{\mathbf{p}}^\delta) - \mathbf{h}^{\delta}\|_{L^{2}(W_1)}\leq \left(C_4\delta^{2} + \alpha\mathcal{M}^{2}\right)^{\frac{1}{2}},
\]
\[
\mathcal{R}(\bar{\mathbf{p}}^\delta)\leq \left(\frac{C_4\delta^{2}}{\alpha} +\mathcal{M}^{2}\right)^{\frac{1}{2}}.
\]
When \(\alpha \sim \delta^{2}\) , choose \(C_{2}\delta^{2}\leq \alpha \leq C_{3}\delta^{2}\) , and then we get
\[
\| G(\bar{\mathbf{p}}^\delta) - \mathbf{h}^{\delta}\|_{L^{2}(W_1)}\leq (C_4 + C_{3}\mathcal{M}^{2})^{\frac{1}{2}}\delta ,
\]
\[
\mathcal{R}(\bar{\mathbf{p}}^\delta)\leq (C_4C_{2}^{-1} + \mathcal{M}^2)^{\frac{1}{2}}.
\]
Therefore
\[
\begin{aligned}
&\|F(q_{\mathrm{recon}})-F(q_{\mathrm{true}})\|_{L^2(W_1)}\\
&\le \|F(q_{\mathrm{recon}})-G(\bar{\mathbf{p}}^\delta)\|_{L^2(W_1)}+\|G(\bar{\mathbf{p}}^\delta)-\mathbf{h}^\delta\|_{L^2(W_1)}+\|\mathbf{h}^\delta-F(q_{\mathrm{true}})\|_{L^2(W_1)}\\
&\le \epsilon_{\mathrm{model}}+(C_4 + C_{3}\mathcal{M}^{2})^{\frac{1}{2}}\delta +\delta\\
&= (\ddot{C}+(C_4 + C_{3}\mathcal{M}^{2})^{\frac{1}{2}}+1)\delta,
\end{aligned}
\]
and by the properties of Sobolev spaces \cite{M-2000}, it can be verified that \(L^{2}(W_{1})\) embeds to \(H^{- s}(W_{1})\) and
\[
\| F(q_{\mathrm{recon}})-F(q_{\mathrm{true}})\|_{H^{-s}(W_{1})}\leq C_{5}(\ddot{C}+(C_4 + C_{3}\mathcal{M}^{2})^{\frac{1}{2}}+1)\delta ,
\]
for a positive constant \(C_{5}\) . By \eqref{eq532},
\[
\| q_{\text{recon}} - q_{\text{true}}\|_{L^{\infty}(\Omega)}\leq \omega (\|  F(q_{\mathrm{recon}})-F(q_{\mathrm{true}})\|_{H^{-s}(W_{1})})\leq \mathcal{O}(\omega (\delta))
\]
holds.
\end{proof}

\begin{rem}
The parameterization error \(\epsilon_{\text{param}}\) arises from representing the true potential \(q_{\text{true}}\) by a finite-dimensional vector \(\bar{\mathbf{p}}\in[0,1]^P\). We assume that the parameterization map \(\mathcal{P}:[0,1]^P \to L^\infty(\Omega)\) has sufficient approximation power, e.g., piecewise constant, spline, or finite element basis functions, such that for any prescribed accuracy \(\epsilon>0\) there exist a sufficiently large \(P\) and a reference parameter \(\bar{\mathbf{p}}_*\) satisfying
\[
\|\mathcal{P}(\bar{\mathbf{p}}_*) - q_{\text{true}}\|_{L^\infty(\Omega)} \le \epsilon_{\text{param}}.
\]
In practice, \(\epsilon_{\text{param}}\) can be made arbitrarily small by increasing the number of parameters (e.g., refining the mesh or adding basis functions). In the theoretical analysis, we assume that the parameterization accuracy is controlled to be of the same order as the observation noise level \(\delta\), i.e.,
\[
\epsilon_{\text{param}} \sim \delta.
\]
During the offline stage we generate \(M\) training samples \(\{\bar{\mathbf{p}}^{(m)}\}\), compute the exact observations \(\mathbf{h}^{(m)} = F(\mathcal{P}(\bar{\mathbf{p}}^{(m)}))\) (after discretization), and then construct the surrogate model \(\mathbf{h}_{\text{pred}}(\bar{\mathbf{p}})\) via POD reduction and RBF interpolation. The surrogate satisfies the error estimate 
\[
\|\mathbf{h}_{\text{pred}}(\bar{\mathbf{p}}) - F(\mathcal{P}(\bar{\mathbf{p}}))\|_{L^2(W_1)} \le \epsilon_{\text{POD}} + \epsilon_{\text{RBF}} + \epsilon_{\text{disc}},
\]
where \(\epsilon_{\text{POD}}\) is the POD truncation error (which can be made arbitrarily small by retaining enough modes), \(\epsilon_{\text{RBF}}\) is the RBF interpolation error (uniformly bounded when training samples are sufficiently dense and the radial basis function has good approximation properties), and \(\epsilon_{\text{disc}}\) is the spatial discretization error (e.g. \(h^\kappa\)). In the theoretical analysis, we assume that the various approximate errors (POD truncation, RBF interpolation, spatial discretization) in the offline stage are all controlled to be much smaller than the level of the observation noise, on the order of $\delta$. This assumption can be approximately satisfied in practice by appropriately increasing the number of retained modes, the number of training samples, and the resolution of the spatial grid
\[
\epsilon_{\text{model}}=\epsilon_{\text{POD}} + \epsilon_{\text{RBF}} + \epsilon_{\text{disc}} \sim \delta.
\]

 Therefore for the inversion we may effectively assume
\[
 \epsilon_{\text{param}}= \dot{C}\delta,\quad \epsilon_{\text{model}} = \ddot{C}\delta.
\]
\end{rem}

\begin{rem}
The above analysis shows that the POD-RBF surrogate method does not degrade the theoretical convergence order of the inversion. Although parameterization and reduced-order approximations are introduced, the final convergence rate remains logarithmic as long as the additional errors are kept below the noise level. Moreover, compared to \cite{Li2025}, the surrogate model greatly reduces the computational cost of the online inversion, making the proposed method attractive for practical applications.
\end{rem}

\subsection{Algorithm}

Because the surrogate model is an explicit analytic function of \(\bar{\mathbf{p}}\), its gradient can be computed exactly. The partial derivative of \(\mathbf{h}_{\mathrm{pred}}\) with respect to the \(j\)-th component is
\begin{equation}
\frac{\partial\mathbf{h}_{\mathrm{pred}}}{\partial\bar{p}_j} = \boldsymbol{\Psi}\mathbf{B}\,\frac{\partial\mathbf{f}}{\partial\bar{p}_j},
\label{eq53}
\end{equation}
where \(\mathbf{f}(\bar{\mathbf{p}}) = [f_1(\bar{\mathbf{p}}),\dots,f_M(\bar{\mathbf{p}})]^T\) with the inverse multiquadric radial basis function
\begin{equation}
f_m(\bar{\mathbf{p}}) = \frac{1}{\sqrt{\|\bar{\mathbf{p}}-\bar{\mathbf{p}}^{(m)}\|_2^2 + \sigma^2}},
\label{eq54}
\end{equation}
and
\begin{equation}
\frac{\partial f_m}{\partial\bar{p}_j} = -\frac{\bar{p}_j - \bar{p}_j^{(m)}}{\bigl(\|\bar{\mathbf{p}}-\bar{\mathbf{p}}^{(m)}\|_2^2 + \sigma^2\bigr)^{3/2}}.
\label{eq55}
\end{equation}
Consequently, the gradient of \eqref{eq51} is
\begin{equation}
\nabla J(\bar{\mathbf{p}}) = \left(\frac{\partial\mathbf{h}_{\mathrm{pred}}}{\partial\bar{\mathbf{p}}}\right)^{\!T}\!\bigl(\mathbf{h}_{\mathrm{pred}}(\bar{\mathbf{p}})-\mathbf{h}^{\delta}\bigr) + \frac{\alpha}{2}\,\nabla\mathcal{R}(\bar{\mathbf{p}}),
\label{eq56}
\end{equation}
where
\[
\nabla\mathcal{R}(\bar{\mathbf{p}}) =
\begin{cases}
2D^{\!T}D\bar{\mathbf{p}}, & \text{1D case},\\[4pt]
2\Delta^{\!T}\Delta\bar{\mathbf{p}}, & \text{2D case}.
\end{cases}
\]

Next, we minimize \eqref{eq51} using a conjugate gradient (CG) method. The iterative procedure is as follows.
We apply the following discrepancy principle,
\begin{equation}
\|\mathbf{h}_{\mathrm{pred}}(\bar{\mathbf{p}}_k) - \mathbf{h}^{\delta}\|_{L^2} \le \tau\delta,
\label{eq57}
\end{equation}
here coefficient \(\tau\ge 1\).


The optimal normalized parameter \(\bar{\mathbf{p}}^*\) is obtained. After inverse normalization to the physical scale,
\begin{equation}
p_i^{\,*} = p_i^{\min} + \bar{p}_i^{\,*}\bigl(p_i^{\max}-p_i^{\min}\bigr),\quad i=1,\dots,P,
\label{eq59}
\end{equation}
and the reconstructed potential is
\begin{equation}
q^*(x) = \sum_{i=1}^{P} p_i^{\,*}\,\phi_i(x).
\label{eq510}
\end{equation}

\noindent
\setlength{\fboxsep}{4pt}
\fbox{\begin{minipage}{0.9\linewidth}
\refstepcounter{algorithm}%
\textbf{Algorithm~\thealgorithm: POD-RBF Conjugate Gradient Inversion}
\\[-8pt]%
\noindent\makebox[\linewidth][c]{\rule{\dimexpr\linewidth+2\fboxsep\relax}{0.4pt}}
\vspace{-10pt}
\begin{algorithmic}
\STATE \textbf{Offline stage (precomputation):}
\STATE \quad Generate \(M\) training samples \(\{\bar{\mathbf{p}}^{(m)}\}_{m=1}^M\).
\STATE \quad For each sample, solve the forward problem \eqref{eq26} (1D) or \eqref{eq215} (2D) to obtain \(\mathbf{U}_R^{(m)}\).
\STATE \quad Compute observation vectors \(\mathbf{h}^{(m)}\) via \eqref{eq217}, form snapshot matrix \(\mathbf{H}\).
\STATE \quad Center snapshots: \(\tilde{\mathbf{h}}^{(m)} = \mathbf{h}^{(m)} - \bar{\mathbf{h}}\) with \(\bar{\mathbf{h}} = \frac1M\sum_{m=1}^M\mathbf{h}^{(m)}\).
\STATE \quad Compute POD basis \(\boldsymbol{\Psi}\) and amplitude matrix \(\mathbf{A}\) (see Section~3).
\STATE \quad Train RBF network to obtain coefficient matrix \(\mathbf{B}\) (see Section~4).
\STATE \textbf{Online stage (for each measurement \(\mathbf{h}^{\delta}\)):}
\STATE \quad Initialize \(\bar{\mathbf{p}}_0 = \mathbf{0}\), \(k=0\).
\STATE \quad Compute \(\mathbf{h}_{\mathrm{pred}}(\bar{\mathbf{p}}_0)\) using \eqref{eq48} and residual \(\mathbf{r}_0 = \mathbf{h}_{\mathrm{pred}}(\bar{\mathbf{p}}_0)-\mathbf{h}^{\delta}\).
\STATE \quad Evaluate \(\mathbf{g}_0 = \nabla J(\bar{\mathbf{p}}_0)\) via \eqref{eq56}, set descent direction \(\mathbf{d}_0 = -\mathbf{g}_0\).
\STATE \textbf{While} \(\|\mathbf{h}_{\mathrm{pred}}(\bar{\mathbf{p}}_k)-\mathbf{h}^{\delta}\|_2 > \tau\delta\) \textbf{do}
\STATE \quad \emph{Line search:} compute step size \(\beta_k\) using the Gauss每Newton approximation
\[
\beta_k = \frac{\mathbf{g}_k^{\!T}\mathbf{d}_k}{\mathbf{d}_k^{\!T}\mathbf{H}_k\mathbf{d}_k + \alpha\,\mathbf{d}_k^{\!T}\nabla^2\mathcal{R}(\bar{\mathbf{p}}_k)\mathbf{d}_k},
\]
\STATE \qquad where \(\mathbf{H}_k = \bigl(\frac{\partial\mathbf{h}_{\mathrm{pred}}}{\partial\bar{\mathbf{p}}}\big|_{\bar{\mathbf{p}}_k}\bigr)^{\!T}\!\bigl(\frac{\partial\mathbf{h}_{\mathrm{pred}}}{\partial\bar{\mathbf{p}}}\big|_{\bar{\mathbf{p}}_k}\bigr)\).
\STATE \quad Update: \(\bar{\mathbf{p}}_{k+1} = \bar{\mathbf{p}}_k + \beta_k\mathbf{d}_k\). Project onto \([0,1]^P\).
\STATE \quad Compute new gradient \(\mathbf{g}_{k+1} = \nabla J(\bar{\mathbf{p}}_{k+1})\).
\STATE \quad Compute conjugate coefficient:
\[
\gamma_{k+1} = \frac{\mathbf{g}_{k+1}^{\!T}(\mathbf{g}_{k+1}-\mathbf{g}_k)}{\|\mathbf{g}_k\|_2^2}, \qquad
\mathbf{d}_{k+1} = -\mathbf{g}_{k+1} + \gamma_{k+1}\mathbf{d}_k.
\]
\STATE \quad \(k \leftarrow k+1\).
\STATE \textbf{End while}
\STATE \quad Output \(\bar{\mathbf{p}}^* = \bar{\mathbf{p}}_k\). Reconstruct physical potential via \eqref{eq510}.

\end{algorithmic}
\end{minipage}}

\section{Numerical experiments}
\label{sec:experiments}
In this section, we present the numerical results for one-dimensional and two-dimensional fractional Calder車n problems. The observed values are given by $$h^\delta = (-\Delta)^s u|_{W}+\delta(2rand(size(h)) - 1),$$ where rand generates random numbers uniformly distributed over [0,1], $size(h)$ is the number of discrete points in $W_2$, and $\delta$ is the level of noise in the observation.

$Example~1.$\ \  In this example, let $\tau=1.01$, the potential \(q\) is given by \(q(x) = \sin (\pi x)\), and we assume that \(\Omega = (-1,1)\), \(s = 0.4\), and \(f(x)\) is a \(\mathbf{1}_{W_1}\) polished function. Let the observation set \(W_2 = (-3, -1 - \epsilon) \cup (1 + \epsilon, 3)\), \(0 < \epsilon \ll 1\), and \(W_1 = W_2\).  Fig. \ref{fig41}(a) presents the reconstruction results of $q^\delta(x)$ with noise level $\delta$=1E-7, 1E-5, 1E-3 compared to the real potential $q(x)$. It shows that the reconstruction performs well at lower noise levels $\delta$, however, the quality of the results decreases rapidly for higher noise, which reflects the pathologies of fractional Calderon problem. The three subplots in Fig. \ref{fig41}(b) respectively illustrate the crucial convergence information during the inversion process based on POD and CG. Left: Decay of POD eigenvalues, the dashed red line indicates the truncation index K; Middle: Cumulative energy percentage as a function of the number of POD modes; Right: Residual norm $||\mathbf{h}_{\mathrm{pred}} - \mathbf{h}^{\delta}||$ versus CG iteration, the red dashed line denotes the discrepancy principle $\tau\delta$.
\begin{figure}[htbp]
	\centering    \captionsetup{font={small}}
	{
		\begin{minipage}{0.7\linewidth}
			\centering
			\includegraphics[height=0.8\columnwidth,width=1\linewidth]{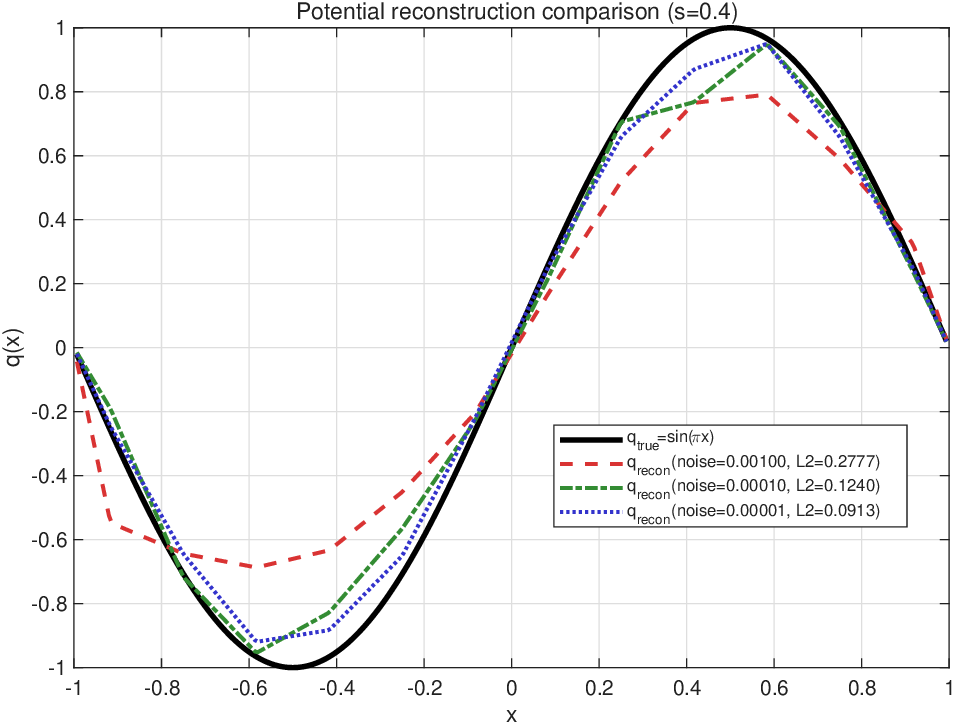}\\
			(a)
		\end{minipage}
	}
\newline
	{
		\begin{minipage}{1\linewidth}
			\centering
			\includegraphics[height=0.5\columnwidth,width=1\linewidth]{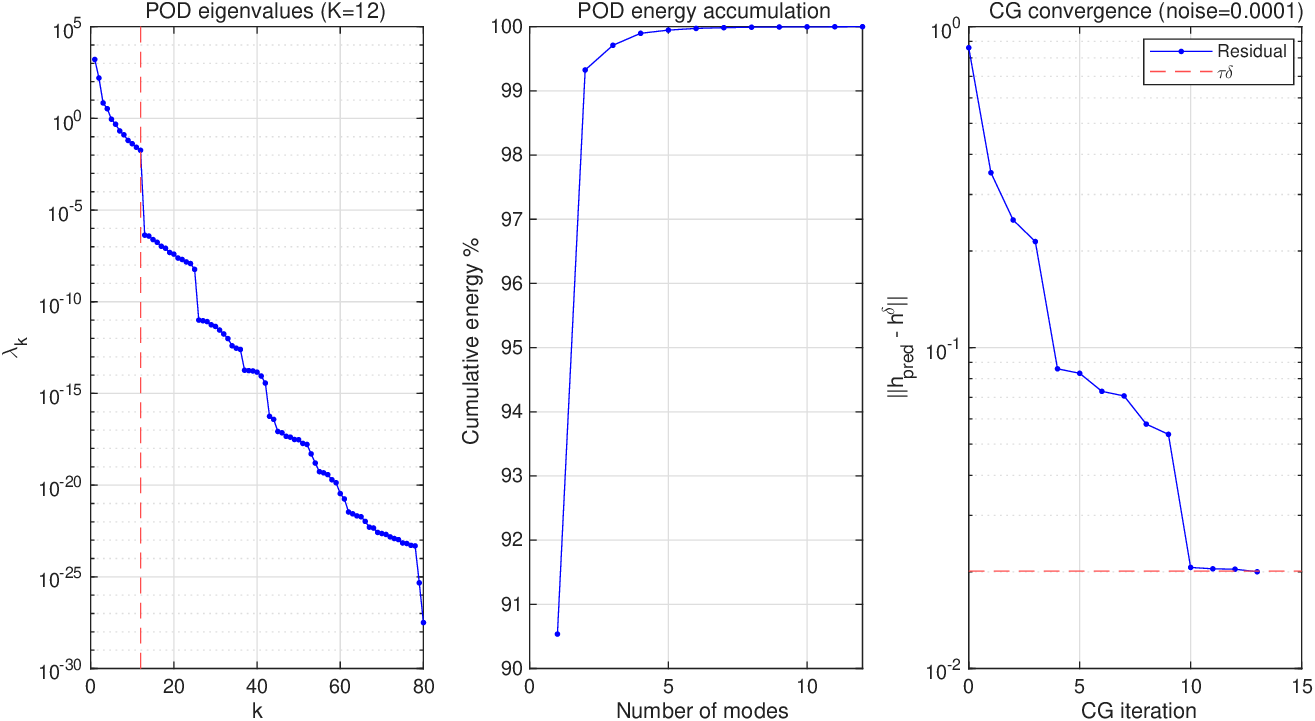}\\
			(b)
		\end{minipage}
	}
\caption{The numerical results for Example 1 with different noise level $\delta$.}
	\label{fig41}
\end{figure}

$Example~2.$\ \  In this example, \(q(x)\) is prescribed as
\[
q(x) = \begin{cases}
10(0.75^2 - x^2)^2, & |x| \le 0.75,\\
0, & \text{otherwise},
\end{cases}
\]
Other parameters such as \(\Omega = (-1,1)\), \(s=0.6\), the boundary datum \(f(x)\) (a smoothed indicator of \(W_1\)), and the observation sets \(W_1 = W_2 = (-3, -1-\epsilon) \cup (1+\epsilon, 3)\) with \(0<\epsilon\ll 1\) are the same as in Example~4.1. The regularization parameter is chosen as \(\alpha \sim \delta^2\).  The inversion potential $q^\delta(x)$ with noise level $\delta$ =1E-7,1E-6,1E-5 and $\alpha=\delta^2$ is shown in Fig. \ref{fig42} along with the real potential $q(x)$.
\begin{figure}[htbp]
	\centering    \captionsetup{font={small}}
	{
		\begin{minipage}{0.7\linewidth}
			\centering
			\includegraphics[height=0.8\columnwidth,width=1\linewidth]{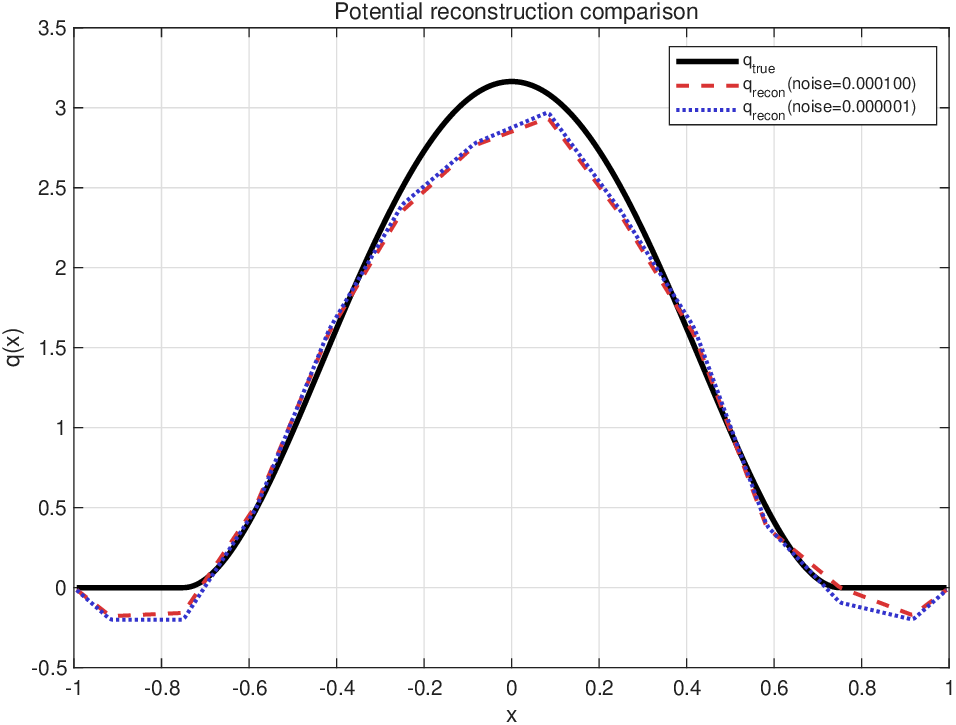}\\
			(a)
		\end{minipage}
	}
\caption{The numerical results for Example 2 with different noise level $\delta$.}
	\label{fig42}
\end{figure}

$Example~3.$\ \
Here we consider a less regular, piecewise constant potential \(q(x)\). The exact expression is not given in the original text. For completeness, the potential is taken as a discontinuous function; a typical choice in such tests is
\[
q(x) =
\begin{cases}
1, & |x| \le 0.5,\\
0, & \text{otherwise},
\end{cases}
\]
though the original paper does not specify the exact form. Let $s=0.8$, other parameters (\(\Omega, f, W_1, W_2\)) are identical to those in Example~4.1. Fig. \ref{fig43} shows the numerical results of the estimated potential using our algorithm at a noise level of $\delta$ = 1E-7, as well as the true potential. Compared with the previous two examples, the error is larger, which is due to the prior smoothing property.
\begin{figure}[htbp]
	\centering    \captionsetup{font={small}}
	{
		\begin{minipage}{0.7\linewidth}
			\centering
			\includegraphics[height=0.8\columnwidth,width=1\linewidth]{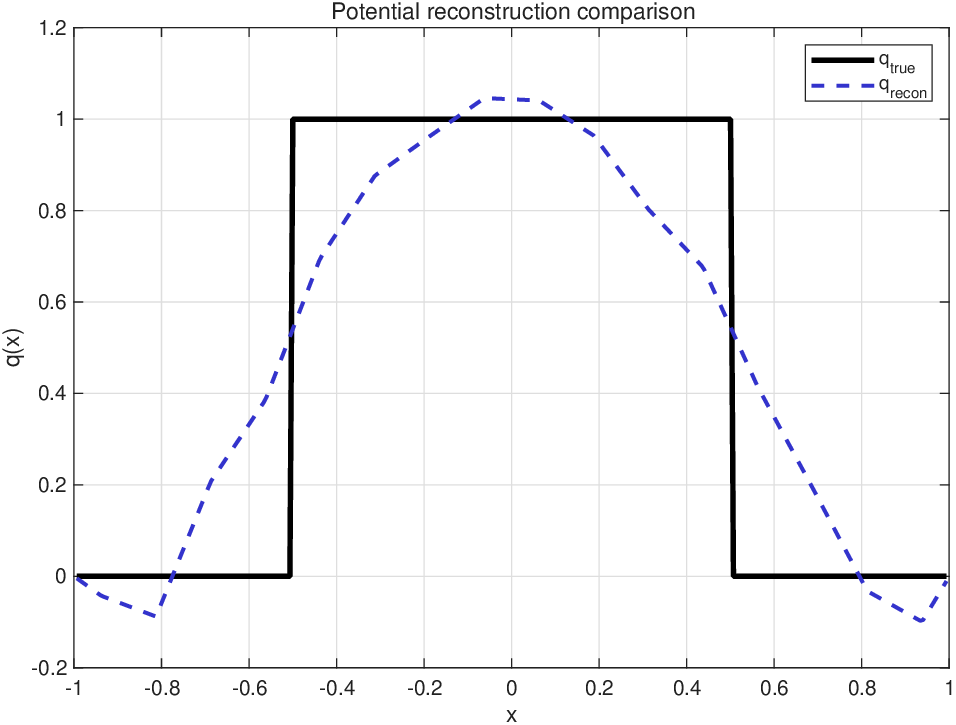}\\
			(a)
		\end{minipage}
	}
\caption{The numerical results for Example 3 with noise level $\delta=1e-7$.}
	\label{fig43}
\end{figure}

$Example~4.$\ \   This example tests the two-dimensional algorithm. The true potential is
\[
q(x,y) = 100 \times \max\Bigl( (0.75^2 - x^2)^3 (0.75^2 - y^2)^3,\; 0 \Bigr),
\]
defined on \(\Omega = (-1,1)^2\). The fractional exponent is \(s=0.5\), and the boundary datum \(f\) is a smoothed indicator of \(W_1\). The observation set is
\[
W_2 = (-3,3)^2 \setminus (-1-\epsilon,1+\epsilon)^2,\qquad 0<\epsilon\ll 1,
\]
with \(W_1 = W_2\). Fig. \ref{fig44}(a) and Fig. \ref{fig44}(b) respectively illustrate the true potential $q(x,y)$ and the reconstructed result $q^\delta(x,y)$, where $\delta$ = 1E-6. Fig. \ref{fig44}(c) shows the absolute error $|q^\delta(x,y)-q(x,y)|$ corresponding to $\delta$= 1E-6.
\begin{figure}[htbp]
	\centering
    \captionsetup{font={small}}
	{
		\begin{minipage}{0.45\linewidth}
			\centering
			\includegraphics[height=0.7\columnwidth,width=0.9\linewidth]{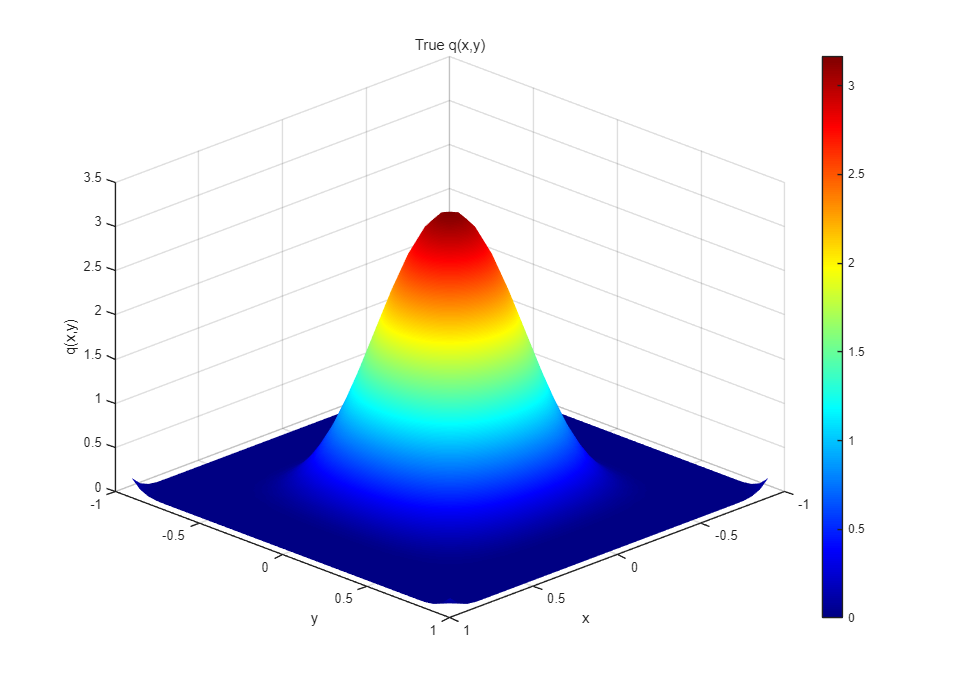}\\
			(a)
		\end{minipage}
		\begin{minipage}{0.45\linewidth}
			\centering
			\includegraphics[height=0.7\columnwidth,width=0.9\linewidth]{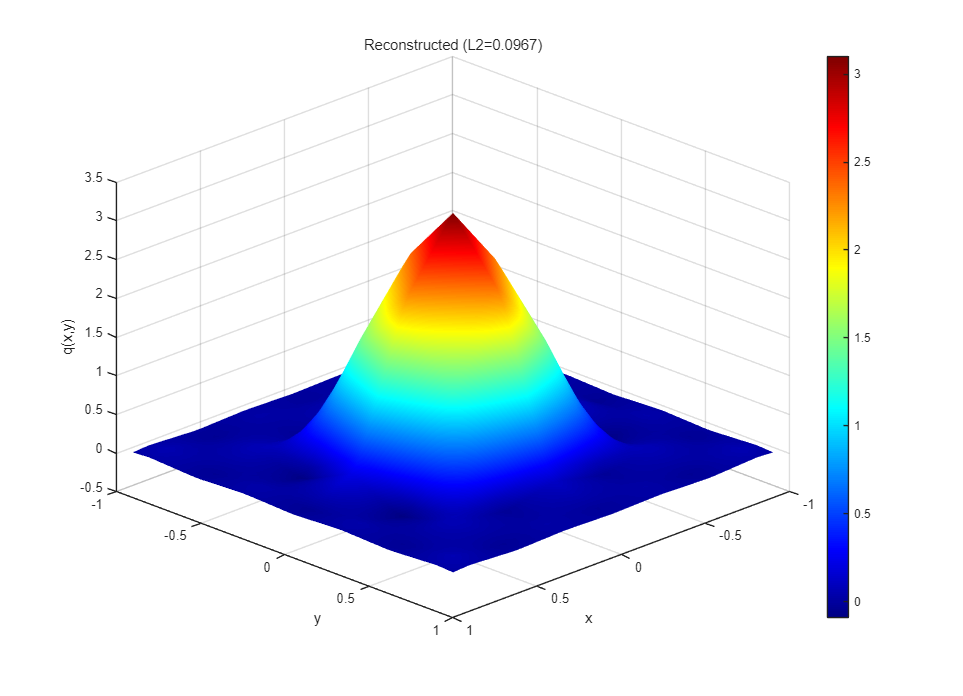}\\
			(b)
		\end{minipage}
	}
\newline
	{
		\begin{minipage}{0.5\linewidth}
			\centering
			\includegraphics[height=0.8\columnwidth,width=1\linewidth]{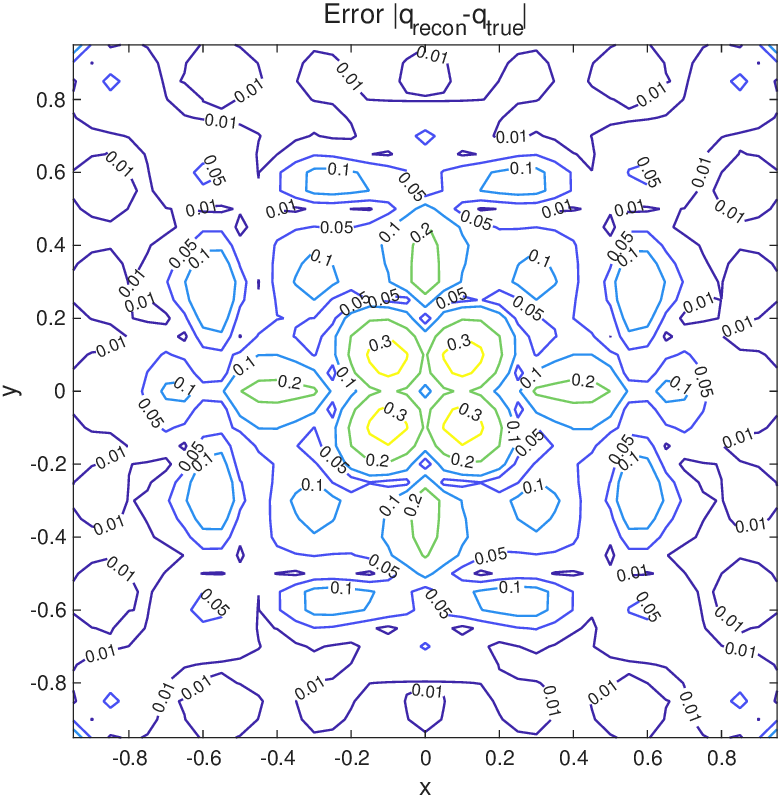}\\
			(c)
		\end{minipage}
	}
\caption{The numerical results of two-dimensional fractional Calder\'{o}n problem for Example 4, (a) the real potential $q(x,y)$; (b) the reconstruction potential $q^\delta(x,y)$ with $\delta$=1E-6; (c) the reconstruction error $|q^\delta(x,y)-q(x,y)|$.}
	\label{fig44}
\end{figure}

\section{Conclusion}
\label{sec:conclusion}

This paper develops a POD-RBF network method tailored for the fractional Calder\'{o}n problem with a single measurement. By constructing a reduced-order surrogate directly from observation snapshots, the method circumvents the repeated solution of the expensive nonlocal forward problem during the inversion iterations. The finite difference method for one- and two-dimensional cases includes the treatment of the extended boundary conditions and the affine observation operator. The POD reduction captures the dominant modes of the measurement variation, while the RBF network provides a smooth nonlinear map from the potential parameters to the modal amplitudes. Coupled with Tikhonov regularization and an efficient gradient-based optimization, the approach achieves high computational efficiency and robustness to noise. Extensions to three-dimensional problems and more complex nonlinear fractional equations are envisioned for future work.

\section*{\textbf{CONFLICT OF INTEREST STATEMENT}}
This work does not have any conflicts of interest.
\section*{\textbf{FUNDING}}
The first author was supported by the NSF of China (12261082). The second author thanks the NSF of China (12261082, 11326234), NSF of Gansu Province (145RJZA099), Scientific research project of Higher School in Gansu Province (2014A-012), Project of NWNU-LKQN2020-08, Graduate First Class Course Construction Project of Northwest Normal University in 2024 (2024KCLX01012) and Graduate Education and Curriculum Reform Project of Northwest Normal University in 2024 (2024KGLX01016). The third author thanks the NSF of China (12501593).
\section*{\textbf{AUTHOR CONTRIBUTIONS}}
Y.-L. Liu: Methodology; software; investigation; formal analysis; writing-original draft; writing-review and editing. J. Wen: Conceptualization; methodology; visualization; software; investigation; formal analysis; supervision; funding acquisition; project administration; writing-original draft; writing-review and editing. J.-L. Fu:  Formal analysis; funding acquisition; project administration; writing-original draft; writing-review and editing.

\bibliographystyle{unsrt}
\bibliography{liu}

\end{document}